\def\draft{n}
\documentclass{amsart}
\usepackage[headings]{fullpage}
\usepackage{amssymb,epic,eepic,epsfig}
\usepackage{amssymb,epic,eepic,epsfig,amsmath,amsbsy,amsmath,bbm}
\usepackage{graphicx}
\usepackage{texdraw}
\usepackage{url}
\usepackage[bookmarks=true,%
    colorlinks=true,%
    linkcolor=blue,%
    citecolor=blue,%
    filecolor=blue,%
    menucolor=blue,%
    urlcolor=blue,%
    breaklinks=true]{hyperref}

\newtheorem{theorem}{Theorem}[section]
\newtheorem{proposition}{Proposition}[section]
\theoremstyle{definition}

\newtheorem{remark}[proposition]{Remark}

\newtheorem{conjecture}[proposition]{Conjecture}

\def\printname#1{
        \if\draft y
                \smash{\makebox[0pt]{\hspace{-0.5in}
                        \raisebox{8pt}{\tt\tiny #1}}}
        \fi
}

\newcommand{\psdraw}[2]
         {\begin{array}{c} \hspace{-1.3mm}
        \raisebox{-4pt}{\epsfig{figure=draws/#1.eps,width=#2}}
        \hspace{-1.9mm}\end{array}}

\newlength{\standardunitlength}
\catcode`\@=11
\long\def\@makecaption#1#2{%
     \vskip 10pt

\setbox\@tempboxa\hbox{
       \small\sf{\bfcaptionfont #1. }\ignorespaces #2}%
     \ifdim \wd\@tempboxa >\captionwidth {%
         \rightskip=\@captionmargin\leftskip=\@captionmargin
         \unhbox\@tempboxa\par}%
       \else
         \hbox to\hsize{\hfil\box\@tempboxa\hfil}%
     \fi}
\font\bfcaptionfont=cmssbx10 scaled \magstephalf
\newdimen\@captionmargin\@captionmargin=2\parindent
\newdimen\captionwidth\captionwidth=\hsize
\catcode`\@=12

\def\lbl#1{\label{#1}\printname{#1}}

\def\BN{\mathbbm N}
\def\BZ{\mathbbm Z}
\def\BQ{\mathbbm Q}
\def\BR{\mathbbm R}
\def\BC{\mathbbm C}

\def\bn{\mathbf{n}}
\def\ba{\mathbf{a}}
\def\bb{\mathbf{b}}
\def\BL{\mathbf{L}}
\def\BM{\mathbf{M}}

\def\l{\lambda}

\def\la{\langle}
\def\ra{\rangle}

\def\e{\epsilon}

\def\th{\theta}

\def\longto{\longrightarrow}

\def\SL{\mathrm{SL}}

\def\om{\omega}

\def\fg{\mathfrak{g}}
\def\fsl{\mathfrak{sl}}

\def\Ann{\mathrm{Ann}}
\def\loc{\mathrm{loc}}
\def\GF{\mathrm{GF}}

\def\supp{\operatorname{supp}}

\begin{document}


\title[The $\fsl_3$ Jones polynomial of the trefoil:
a case study of $q$-holonomic sequences]{
The $\fsl_3$ Jones polynomial of the trefoil:
a case study of $q$-holonomic sequences}
\author{Stavros Garoufalidis}
\address{School of Mathematics \\
         Georgia Institute of Technology \\
         Atlanta, GA 30332-0160, USA \newline 
         {\tt \url{http://www.math.gatech.edu/~stavros }}}
\email{stavros@math.gatech.edu}
\author{Christoph Koutschan}
\address{Research Institute for Symbolic Computation \\
         Johannes Kepler University \\
         Altenberger Stra\ss e 69 \\
         A-4040 Linz, Austria \newline
         {\tt \url{http://www.risc.jku.at/home/ckoutsch}}}
\thanks{C.K. was supported by grants DMS-0070567 of the US National Science 
Foundation
and FWF P20162-N18 of the Austrian Science Foundation.
S.G. was supported in part by grant DMS-0805078 of the US National Science 
Foundation. 
\\
\newline
1991 {\em Mathematics Classification.} Primary 57N10. Secondary 57M25.
}

\date{March 1, 2011} 


\begin{abstract}
The $\fsl_3$ colored Jones polynomial of the trefoil knot is a $q$-holonomic
sequence of two variables with natural origin, namely quantum topology.
The paper presents an explicit set of generators for the annihilator ideal of 
this $q$-holonomic sequence as a case study.
On the one hand, our results are new and useful to quantum topology: 
this is the first example of a rank 2 Lie algebra computation 
concerning the colored Jones polynomial of a knot. On the other hand, this
work illustrates the applicability and computational power of the employed 
computer algebra methods.
\end{abstract}

\maketitle

\tableofcontents

\section{The colored Jones polynomial: a case study of $q$-holonomic 
sequences}
\lbl{sec.intro}

\subsection{Introduction}
\lbl{sub.introduction}

The aim of this paper is to investigate the $\fsl_3$ colored Jones polynomial 
of the trefoil knot which is a $q$-holonomic sequence in two variables.
Such sequences arise naturally in quantum topology when studying knots.
The relevance of our results within this field become evident when taking
into account that this is the first example of a rank 2 Lie algebra 
computation concerning the colored Jones polynomial of a knot, in accordance
to the $\fsl_3$ AJ conjecture of \cite{Ga1}.

Using computer algebra, we compute an explicit (conjectured) set of 
generators for this $q$-holonomic sequence as a case study. We employ
the Mathematica packages developed by M. Kauers (see \cite{Ka1,Ka2} 
and the second author (see \cite{Ko2,Ko3,Ko4}). These computations
show the power of the underlying computer algebra methods and may
serve well the theoretical and  practical needs of quantum topology.

Since quantum topology and computer algebra are rather disjoint subjects, 
we need to briefly review some basic concepts before we present our results.

\subsection{$q$-holonomic sequences}
\lbl{sub.qholo}

A $q$-{\em holonomic sequence} $(f_n(q))$ for $ n \in \BN$ 
is a sequence (typically of rational functions $f_n(q) \in \BQ(q)$ in 
one variable $q$)
which satisfies a {\em linear recursion relation}:
$$
a_d(q^n,q) f_{n+d}(q) + \dots + a_0(q^n,q) f_n(q)=0
$$
for all $n \in \BN$, where $a_j(u,v) \in \BQ[u,v]$ for all $j=0,\dots,d$. 
The algorithmic significance of such sequences was first recognized 
by Zeilberger in \cite{Z}, where
also a generalization to multivariate sequences $f_{{\bf n}}(q)$ for 
${\bf n}=(n_1,\dots, n_r) \in \BN^r$ was given. A down-to-earth introduction
of ($q$-) holonomic sequences and closely related notions
is given in~\cite{Ko1}. It is well known that 
$q$-holonomic sequences enjoy several useful properties:

\begin{itemize}
\item[(a)] Addition, multiplication, and specialization
preserve $q$-holonomicity.
\item[(b)] Finite (multi-dimensional) sums of proper $q$-hypergeometric terms
are $q$-holonomic in the remaining free variables. This is the 
{\em fundamental theorem} of WZ theory; see \cite{WZ}.
\item[(c)] The fundamental theorem is constructive, and computer-implemented;
see \cite{PR1,PR2,Ko3,PWZ}. 
\end{itemize}

\subsection{The colored Jones polynomial is $q$-holonomic}
\lbl{sub.cj}

Five years ago, a natural source of $q$-holonomic sequences was discovered:
{\em Knot Theory and Quantum Topology}; see \cite{GL}. Let us review some
basic concepts of knot theory and quantum topology here. For a detailed 
discussion, the reader may look into \cite{B-N,GL,Ja,Jo,Tu1,Tu2}.
A {\em knot} $K$ in 3-space is
a smoothly embedded circle in the 3-sphere $S^3$ considered up to isotopy.
The simplest non-trivial knot is the {\em trefoil}:
$$
\psdraw{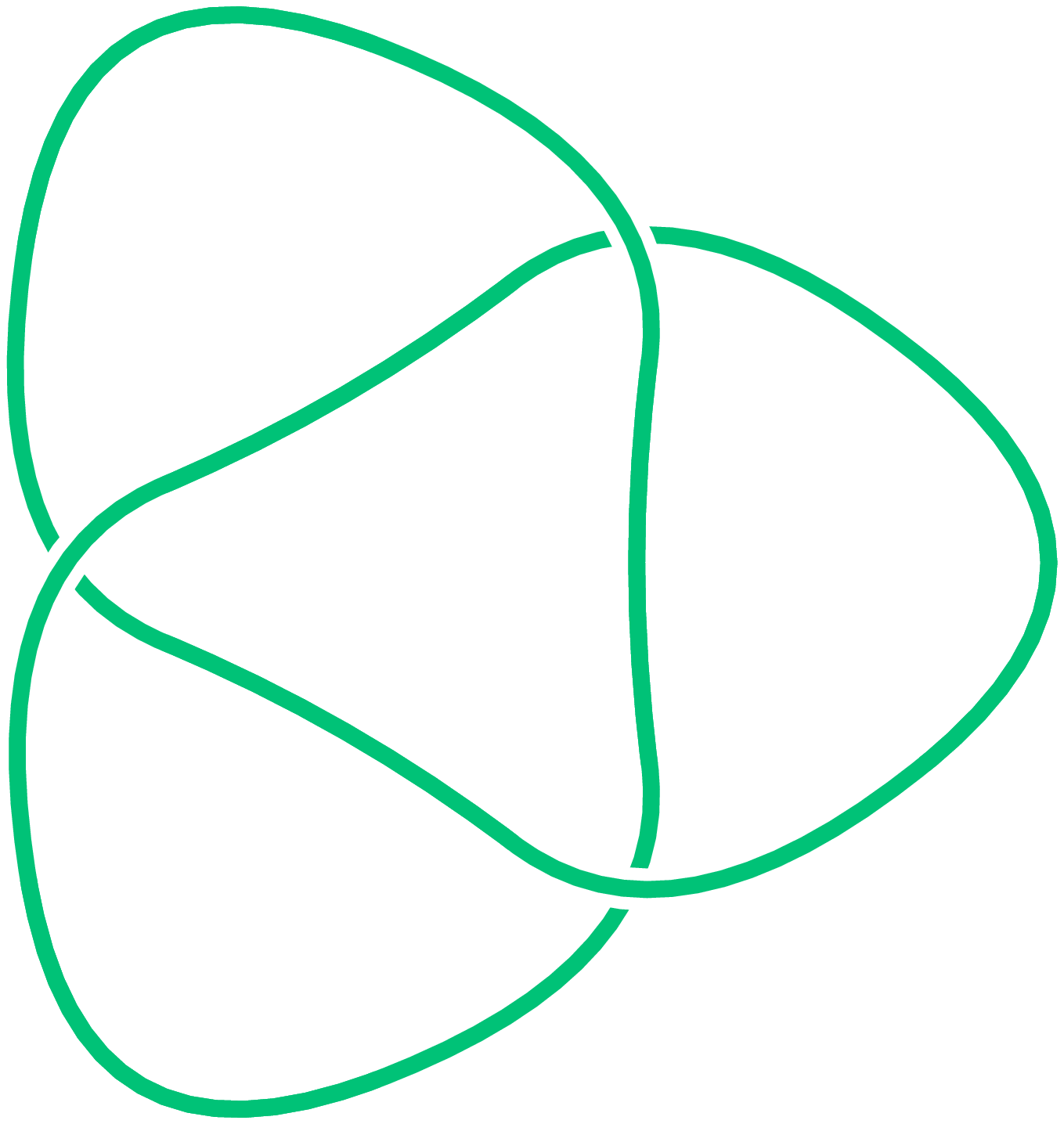}{1in}
$$
In his seminal paper \cite{Jo}, Jones introduced the {\em Jones polynomial}
of a knot $K$. The Jones polynomial of a knot~$K$ is an element
of $\BZ[q^{\pm 1}]=\BZ[q,q^{-1}]$ and can be 
generalized to a polynomial invariant $J_{K,V}(q) \in \BZ[q^{\pm 1}]$ 
of a framed knot $K$ colored by a representation $V$ of a simple Lie 
algebra $\mathfrak{g}$. When the Lie algebra is $\fsl_2$ and $V=\BC^2$
is the 2-dimensional irreducible representation, $J_{K,\BC^2}(q)$ is
the classical Jones polynomial. The definition of $J_{K,V}(q)$ uses the 
machinery of {\em quantum groups} and may be found in \cite{Ja,Tu1,Tu2}.
Without going into details, let us point out some features of this theory. 
Irreducible representations of a simple Lie algebra $\fg$ are parametrized by
dominant weights, and once we choose a basis of fundamental weights, the set 
of dominant weights can be identified with $\BN^r$, where $r$ is the rank 
of the Lie algebra. Thus, the $\fg$ colored Jones polynomial
of a knot is a function ${\bf n} \in \BN^r \mapsto J_{K,{\bf n}}(q) \in 
\BZ[q^{\pm 1}]$. In \cite{GL} the following result was shown.

\begin{theorem}
\lbl{thm.GL}\cite[Thm.1]{GL}
For every simple Lie algebra $\fg$ of rank $r$, and every knot $K$, the 
function $J_{K,\bullet}(q): \BN^r \to \BZ[q^{\pm 1}]$ is $q$-holonomic. 
\end{theorem}
It will be useful to recall in brief the proof of the above 
theorem using principles of $q$-holonomic sequences and 
quantum group theory. The quantum group invariant $J_{K,V}(q)$ 
of a knot $K$ is a local object, obtained by 
\begin{itemize}
\item
choosing a planar projection of the knot, 
\item
assigning a tensor (the so-called $R$-matrix or its inverse, which depends
on the representation $V$ of the simple Lie algebra $\fg$)
to each positive or negative crossing,  
\item
summing over all contractions of indices, and adjust for the framing. 
\end{itemize}
It follows that $J_{K,V}(q)$ is a finite multi-dimensional sum where the 
summand is a product of entries of the $R$-matrix with certain framing 
factors. In \cite{GL} it was shown that
if we choose a basis of the quantum group suitably, the entries of the 
$R$-matrix are $q$-proper hypergeometric multisums. In addition, the 
framing factor is a monomial in $q$ raised to a quadratic 
form in the summation variables. It follows from property (b) of Section
\ref{sub.qholo} that $J_{K,V}(q)$ is $q$-holonomic.
This summarizes the proof of Theorem \ref{thm.GL}.

The proof of Theorem \ref{thm.GL} is algorithmic. When the Lie algebra is 
$\fsl_2$ and $K$ is a knot with a planar projection with $d$ crossings, 
there is an explicit proper $q$-hypergeometric term in $d+1$ variables
$(k_1,\dots,k_d,n)$ such that summation in all $(k_1,\dots,k_d) \in \BN^d$
(this is a finite sum) gives $J_{K,n}(q)$. The algorithm is discussed in detail
in \cite[Sec.3]{GL} and has been computer-implemented in the {\tt KnotAtlas}
via the command {\tt ColouredJones}; see \cite{B-N}. 
The above presentation shows the limitation of the WZ method. Although 
$J_{K,n}(q)$ is a $q$-holonomic sequence presented by an explicit 
$q$-hypergeometric multisum, when $K$ is a knot with several (eg. 8) 
crossings, the computation of a recursion relation for $J_{K,n}(q)$
via WZ theory of \cite{WZ} becomes impractical.

\subsection{The $\fsl_3$-colored Jones polynomial of the trefoil}
\lbl{sub.cjtrefoil}

Concrete formulas for the colored Jones polynomial $J_{K,V}(q)$ are hard to 
find in the case of higher rank Lie algebras, and for good reasons. 
For {\em torus knots}~$T$, Jones and Rosso gave a formula for $J_{T,V}(q)$ which 
involves a {\em plethysm map} of $V$, unknown in general; see \cite{JR}.
The plethysm map was computed explicitly by \cite{GV1} for $\fsl_3$
(and in forthcoming work of \cite{GV2} for all simple Lie algebras). 
This gives an explicit formula for the trefoil, and more generally for the
$T(2,b)$ torus knots, where $b$ is an odd natural number. Let us recall this
formula, which is the starting point of our paper, using the notation
of \cite{GV1}. Let 
\begin{equation}
\lbl{eq.fm1m2}
f_{b,n_1,n_2}(q)=J_{T(2,b),n_1,n_2}(q)
\end{equation}
denote the $\fsl_3$ quantum group invariant of the torus knot $T(2,b)$
colored with the irreducible representation $V_{n_1,n_2}$ of $\fsl_3$
of highest weight $\l=n_1 \om_1 + n_2 \om_2$ ($0$-framed and normalized to be 
$1$ at the unknot), where $n_1, n_2$ are non-negative natural numbers and
$\om_1,\om_2$ are the fundamental weights of $\fsl_3$; \cite{FH,Hu}.

\begin{theorem}
\lbl{thm.GV}\cite{GV1}
For all odd natural numbers $b$ we have
\begin{eqnarray*}
f_{b,n_1,n_2}(q) &=& \frac{\theta_{n_1,n_2}^{-2b}}{d_{n_1,n_2}}
\left(
\sum_{l=0}^{\min\{n_1,n_2\}}\sum_{k=0}^{n_1-l}(-1)^k d_{2n_1-2k-2l,2n_2+k-2l}
\theta_{2n_1-2k-2l,2n_2+k-2l}^{\frac{b}{2}} \right. \\ 
&   & \qquad\qquad 
+  \sum_{l=0}^{min\{n_1,n_2\}}\sum_{k=0}^{n_2-l}(-1)^k d_{2n_1+k-2l,2n_2-2k-2l}
\theta_{2n_1+k-2l,2n_2-2k-2l}^{\frac{b}{2}} \\
&   & \qquad\qquad - \left. \sum_{l=0}^{min\{n_1,n_2\}}  d_{2n_1-2l,2n_2-2l}
\theta_{2n_1-2l,2n_2-2l}^{\frac{b}{2}} \right)
\end{eqnarray*}
where the {\em quantum integer} $[n]$, 
the {\em quantum dimension} $d_{n_1,n_2}$ and the 
{\em twist parameter} $\th_{n_1,n_2}$ of $V_{n_1,n_2}$ are defined by
\begin{eqnarray}
\lbl{eq.n}
[n]&=&\frac{q^{\frac{n}{2}}-q^{-\frac{n}{2}}}{q^{\frac{1}{2}}-q^{-\frac{1}{2}}}
\\
\lbl{eq.dim}
d_{n_1,n_2}&=&\frac{[n_1+1][n_2+1][n_1+n_2+2]}{[2]}
\\
\lbl{eq.theta}
\theta_{n_1,n_2}&=&q^{\frac{1}{3}(n_1^2+n_1n_2+n_2^2)+n_1+n_2}
\end{eqnarray}
\end{theorem}

\subsection{The $q$-Weyl algebra}
\lbl{sub.qweyl}

Our results below are phrased in terms of the $q$-Weyl algebra, and its
corresponding module theory and geometry. Let us recall the appropriate 
$q$-Weyl algebra for bivariate $q$-holonomic sequences.

Consider the operators $M_i, L_i$
for $i=1,2$ which act on a sequence $f_{n_1,n_2}(q) \in \BQ(q)$ by:
\begin{align*}
(L_1 f)_{n_1,n_2}(q)& = f_{n_1+1,n_2}(q) & (M_1 f)_{n_1,n_2}(q)& = q^{n_1} f_{n_1,n_2}(q)
\\ 
(L_2 f)_{n_1,n_2}(q)& = f_{n_1,n_2+1}(q) & (M_2 f)_{n_1,n_2}(q)& = q^{n_2} f_{n_1,n_2}(q)
\end{align*}
It is easy to see that the operators $q,M_1,L_1,M_2,L_2$ commute except in
the following cases:
\begin{align*}
 L_1 M_1 & = q M_1 L_1 & L_2 M_2 & = q M_2 L_2
\end{align*}
It follows that if $R(q,M_1,M_2) \in \BQ(q,M_1,M_2)$ is a rational function,
then we have the commutation relations
\[
\lbl{eq.I}
  I=\{L_1 R(q,M_1,M_2)-R(q,q M_1,M_2)L_1, L_2 R(q,M_1,M_2)-R(q,M_1,q M_2)L_2\}.
\]
Consider the (localized) $q$-{\em Weyl algebra} 
\begin{equation}
\lbl{eq.qweyl}
W_{q,\loc}= \BQ(q,M_1,M_2)\la L_1, L_2 \ra/(I) 
\end{equation}
generated by the operators $L_1,L_2$ over the field $\BQ(q,M_1,M_2)$,
modulo the 2-sided ideal $I$ above. $W_{q,\loc}$ is an example of an
{\em Ore extension} of the field $\BQ(q,M_1,M_2)$; see \cite{Or}.
Given a bivariate
sequence $f_{n_1,n_2}(q) \in \BQ(q)$, consider the annihilating ideal
$$
\Ann_f=\{P \in W_q \, | \, Pf=0 \}.
$$
$\Ann_f$ is always a left ideal in $W_q$, which gives rise to the cyclic
$W_{q,\loc}$-module $M_f=W_{q,\loc}/\Ann_f$. Finitely generated $W_{q,\loc}$-modules 
$M$ have a 
well-defined theory of {\em dimension} (and {\em Hilbert polynomial}), 
and satisfy the key {\em Bernstein inequality}; see \cite[Thm.2.1.1]{Sa}. 
Thus, one can
define $q$-holonomic modules. When $M=W_{q,\loc}/J$ for a left ideal $J$, the 
Hilbert polynomial of the left ideal $J$ can be computed by 
{\em Gr\"obner bases}~\cite{Bu} and their {\em Gr\"obner fans}
which are well-defined for the case of the Ore
algebra $W_{q,\loc}$ under consideration; see \cite{Sa,Ko1}. For a computer 
implementation of Gr\"obner bases in the Ore algebra $W_{q,\loc}$, see 
\cite{CS,Ko1}.

\subsection{Our results}
\lbl{sub.results}

One may try to compute generators for the ideal of recursion
relations using techniques which are well known in computer
algebra, but which are recalled in Section~\ref{sec.thm1} for sake of 
self-containedness. Given elements $P_i$ for $i \in S$
of $W_{q,\loc}$, let $\la P_i | i \in S\ra$ denote the left ideal of $W_{q,\loc}$
that they generate. 
Furthermore, let $\tau$ denote the symmetry map that exchanges $n_1$ and~$n_2$,
i.e., for $P\in W_{q,\loc}$ we have $\tau(P(M_1,M_2,L_1,L_2))=P(M_2,M_1,L_2,L_1)$.

\begin{theorem}
\lbl{thm.1}
\rm{(a)}
Let $P_1,P_2,P_3$ be the operators of Appendix \ref{sec.app1}
and 
\begin{equation}
\lbl{eq.IP}
J=\la P_1, P_2, P_3 \ra.
\end{equation}
Then $\{P_1,P_2,P_3\}$ is a Gr\"obner basis for $J$ with respect to
total degree lexicographic order ($L_1>L_2$), 
and $J$ is a zero-dimensional ideal of rank~$5$. 
\newline
\rm{(b)}
We have
$$
J=\la Q_1, Q_2 \ra
$$
where $\{Q_1,Q_2\}$ is a Gr\"obner basis with respect to the lexicographic 
order ($L_1>L_2$). The five monomials below the staircases of 
$\la P_1, P_2, P_3 \ra$ and $\la Q_1, Q_2 \ra$ are given by:
$$
\psdraw{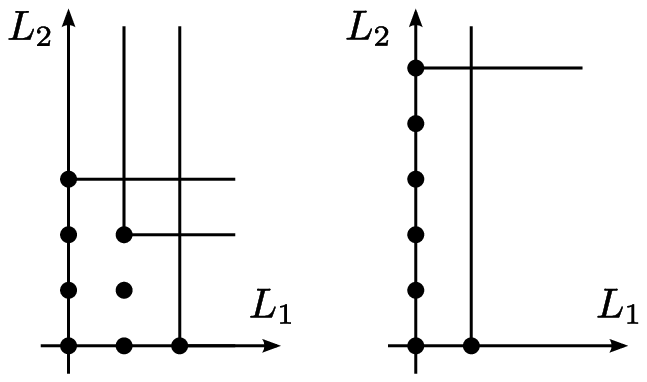}{1.5in}
$$
\newline
\rm{(c)}
The Gr\"obner fan $\GF(J)$ of $J$ is a fan in $\BR^2_+$ 
with rays generated by the vectors $(4,1),(2,1),(1,1),(1,2),(1,4)$
on $(L_1,L_2)$ coordinates.
$$
\psdraw{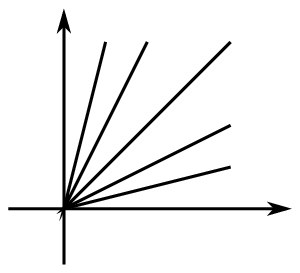}{1in}
$$
\newline
\rm{(d)}
The left ideal~$J$ is invariant under the symmetry map~$\tau$.
The operator~$P_1$ is itself symmetric (modulo sign change)
\end{theorem}

\begin{proof}
The computer proof of statements (a)--(d) is given in the electronic 
supplementary material~\cite{Ko4}.
\end{proof}

\begin{conjecture}
Let~$J$ be the left ideal defined in~\eqref{eq.IP}.
Then 
$$
\Ann_{f_{3,n_1,n_2}}=J.
$$
\end{conjecture}

\begin{remark}
\lbl{rem.symmetry}
The colored Jones polynomial satisfies the symmetry $J_{K,V^*}(q)=J_{K,V}(q)$
where $V^*$ is the dual representation of the simple Lie algebra. 
Since $V_{n_1,n_2}^*=V_{n_2,n_1}$ for the representations of $\fsl_3$, it follows 
that $f_{b,n_1,n_2}(q)=f_{b,n_2,n_1}(q)$.
Thus, the fact that $J$ is invariant under $\tau$ gives additional
evidence for the above conjecture.
\end{remark}


\subsection{The $\fsl_3$ AJ Conjecture for the trefoil}
\lbl{sub.AJ}

One of the best-known geometric invariants of holonomic $D$-modules are
their characteristic varieties, \cite{Ks}. In the case of holonomic $D$-modules
over the complex torus $(\BC^*)^r$, the characteristic variety is a Lagrangian
subvariety of the cotangent bundle $T^*((\BC^*)^r)$, namely a conormal bundle.
$q$-holonomic $D$-modules have a characteristic variety, too. In the case
of the $q$-holonomic $D$-modules that come from knot theory, namely 
the $\fg$ colored Jones polynomial of a knot $K$, their characteristic
variety is a subvariety of $\BC^{2r}$, where $r$ is the rank of the simple
Lie algebra $\fg$. In \cite{Ga1} the first author formulated the AJ
Conjecture which identifies the characteristic variety of the colored
Jones polynomial of a knot with its deformation variety. The latter is
the variety of $G$-representations of the boundary torus of the knot 
complement that extend to $G$-representations of the knot complement,
where $G$ is the simple, simply connected Lie group with Lie algebra $\fg$.

In down-to-earth terms, the $\fsl_3$ AJ Conjecture involves the image of the
$\fsl_3$ recursion ideal of a knot under the partially defined map:
\begin{equation}
\lbl{eq.eWq}
\e: W_{q,\loc} \longto \BQ(M_1,M_2)[L_1,L_2], \qquad e(P)=P|_{q=1}
\end{equation}
For the ideal $J$ of Proposition~\ref{thm.1}, we have
\begin{equation}
\lbl{eq.eJ}
\e(J)=\la p_1,p_2,p_3 \ra
\end{equation}
where $p_1,p_2,p_3$ are given in the appendix.
In general, it is hard to compute the variety of $\SL_3(\BC)$ representations
of a knot complement and to compare it with the above image. For the trefoil
knot, however such a computation is possible (see \cite{Ga2}). It turns out
that $\e(J)$ coincides with the $\SL_3(\BC)$-deformation variety of the trefoil,
thus confirming the AJ conjecture.

\section{The computation}
\lbl{sec.thm1}

\subsection{Guessing generators for the annihilator ideal}
\lbl{sub.guess}

The method of guessing is nothing else but an ansatz
with undetermined coefficients. Assume that a multivariate sequence
$(f_{\bn}(q)),\bn\in\BN^r$, of rational functions in~$q$, i.e.,
$f_{\bn}(q)\in\BQ(q)$, is given by some expression that
allows to compute the values of $(f_{\bn}(q))$ for small~$\bn$, e.g.,
for all~$\bn$ inside the hypercube~$[0,n_0]^r$.  For finding linear
recurrences with polynomial coefficients (they are guaranteed to exist
if $(f_{\bn}(q))$ is known to be $q$-holonomic), an ansatz of the
following form is used:
\begin{equation}\label{ansatz}
  \bigg(\sum_{(\ba,\bb)\in S\subseteq\BN^{2r}} c_{\ba,\bb}\BM^{\ba}\BL^{\bb}\bigg) f_{\bn}(q)=0
\end{equation}
where the {\em structure set}~$S$ is finite and the unknowns~$c_{\ba,\bb}$ are to
be determined in~$\BQ(q)$. The multi-index notation 
$\BM^{\ba}=M_1^{a_1}\cdots M_r^{a_r}$ is employed here.
Linear equations for the $c_{\ba,\bb}$ are
obtained by substituting concrete integer tuples for~$\bn$
into~\eqref{ansatz} and by plugging in the known values of
$(f_{\bn}(q))$. If sufficiently many equations are generated this way,
i.e., if the resulting system of equations is overdetermined, then it
is to be expected with high probability that the solutions of this
system are indeed valid recurrences for the sequence in question.

In the case of a univariate sequence $(f_n(q))$, usually the recursion 
relation with minimal order is desired. The analogue in the multivariate case
is a set of recurrence operators that generate a left ideal in $W_{q,\loc}$
of minimal rank. For their many nice properties it is natural
to aim at a {\em Gr\"obner basis}~\cite{Bu} of this annihilating ideal. There 
are two strategies how to obtain such a basis of recurrences. The first
strategy consists in finding a bunch of recurrences by guessing and
then apply Buchberger's algorithm to obtain a Gr\"obner basis of the
ideal they generate. The more recurrences are used as input, the higher
is the probability of ending up with the full annihilating ideal (the
one with minimal rank). Since this latter computation can be quite 
tedious, we follow a different strategy that doesn't require Buchberger's
algorithm: the structure set~$S$ in the ansatz~\eqref{ansatz}
is adjusted by trial and error until the linear system for the $c_{\ba,\bb}$
admits a single solution, which (supposedly)
is an element of the desired Gr\"obner basis. The latter is achieved
by searching systematically (in the same fashion as in the
FGLM algorithm) through the monomials $\BL^{\bb}$ that appear in the
ansatz. This second strategy will finally deliver a set of operators that 
are very likely to be elements of the annihilator of $(f_{\bn}(q))$
(as it is the case for the first strategy), and that are very likely
to form a Gr\"obner basis (which is not the case for the first strategy).

The above described method involves a lot of trials until the
structure set for each element of the Gr\"obner basis has been
figured out exactly. In order to make this reasonably fast, modular
techniques are used.  First, the input data $(f_{\bn}(q))$ for
$\bn\in[0,n_0]^r$ is computed only for a specific choice of~$q$, say
$q=64$, and modulo some prime number, say $p=2147483647$. 
While in other examples of this type it doesn't really matter which value 
is substituted for~$q$ (there may be a small finite set of non-eligible
integers only), the situation here is slightly different and the choice
$q=64$ not at random: note the rational numbers $\frac12$ and $\frac13$
that appear in the exponents of Theorem~\ref{thm.GV}. Although in the
end all fractional powers disappear, yielding a Laurent polynomial,
the do appear in intermediate expressions. Thus setting $q=m^6$ for some 
$m\in\BN$ simplifies our job significantly. All
following computations (constructing and solving the linear system) are
now done with the homomorphic images in $\BZ_p$ instead of
$\BQ(q)$. The computational complexity is reduced drastically since no
expression swell can happen. Once the exact shape of~\eqref{ansatz} is
found this way, the final computation over $\BQ(q)$ can be started
with a refined ansatz by omitting all unknowns $c_{\ba,\bb}$ that
become zero in the modular computation. This methodology of guessing
as described above is implemented in the Mathematica package 
\texttt{Guess.m} (see~\cite{Ka1,Ka2}).

In our case study of the $\fsl_3$ colored Jones polynomial of the trefoil knot 
we have $r=2$ and Theorem~\ref{thm.GV} allows to compute 
$f_{3,n_1,n_2}(q)$ for sufficiently many $(n_1,n_2)\in\BN^2$.
Trial and error gave evidence that the $(L_1,L_2)$-supports
of the Gr\"obner basis elements are
\begin{eqnarray*}
\supp_{\BL}(P_1) & = & \{L_1^2,L_1L_2,L_2^2,L_1,L_2,1\}\\
\supp_{\BL}(P_2) & = & \{L_2^3,L_1L_2,L_2^2,L_1,L_2,1\}\\
\supp_{\BL}(P_3) & = & \{L_1L_2^2,L_1L_2,L_2^2,L_1,L_2,1\}.
\end{eqnarray*}
As stated in Proposition~\ref{thm.1}, the ideal generated by
these three elements has rank~$5$ and the monomials that cannot
be reduced by $P_1,P_2,P_3$ are the following:
\[
  U = \{L_1L_2,L_2^2,L_1,L_2,1\}.
\]
We tried to guess an operator with support~$U$ and total 
degree~$60$ with respect to $M_1$ and $M_2$ in its coefficients, 
but did not succeed. 
This indicates that we found the ideal of minimal rank, 
taking into account that $P_1,P_2,P_3$ have $(M_1,M_2)$-total 
degrees~$23$, $28$, and $27$, respectively.
The modular computations involved solving linear systems with 
several thousand unknowns, which is not a big deal with today's
hardware. The details of this section are given in the Mathematica notebook
\texttt{GuessIdeal.CJ.2.3.nb} of \cite{Ko4}.

\subsection{Gr\"obner bases in the $q$-Weyl algebra}
\lbl{sub.gb}

Once a Gr\"obner basis with respect to some specified term order is
found, the FGLM algorithm (see \cite{FGLM})
can be employed to obtain a Gr\"obner basis
of the same ideal with respect to any other term order. Using the FGLM
implementation for noncommutative operator algebras in the Mathematica
package \texttt{HolonomicFunctions.m} (see~\cite{Ko2,Ko3}), the second
basis given in Proposition~\ref{thm.1} was computed. Fixing a certain term
order destroys the symmetry with respect to~$n_1$ and~$n_2$ that is
inherent in the problem of the $\fsl_3$-colored Jones polynomials,
because either $L_1>L_2$ or vice versa. Therefore the generators
$P_1,P_2,P_3$ are not expected to be symmetric themselves ($P_1$ however is).  
The symmetry of the problem is revealed when the Gr\"obner fan is
computed.  The resulting universal Gr\"obner basis is symmetric with
respect to~$n_1$ and~$n_2$. Moreover, specific symmetric recurrences
in the ideal can be found by calling the command \texttt{FindRelation}
(\cite{Ko2,Ko3}). For example, by specifying the support 
$\{1,L_1L_2,L_1^2L_2^2,L_1^3L_2^3,L_1^4L_2^4\}$ we found an
operator with this support and coefficients that are symmetric
with respect to $M_1$ and $M_2$. This operator gives rise to a
recurrence for the diagonal sequence~$f_{3,n,n}(q)$.
The details of this section are given in the Mathematica notebook
\texttt{GroebnerBasis.CJ.2.3.nb} of~\cite{Ko4}.

\subsection{Changing the normalization of the colored Jones polynomial}
\lbl{sub.norm}

In quantum topology, one often uses other normalizations
of the colored Jones function of a knot. Theorem \ref{thm.GV} uses the
normalization where the 0-framed unknot has colored Jones polynomial $1$.
If we use the TQFT normalization, where the colored Jones of the unknot
has value $d_{n_1,n_2}$ and an arbitrary fixed framing 
$c \in \BZ$, the corresponding colored Jones function is given by
$$
F_{b,c,n_1,n_2}(q) = d_{n_1,n_2} \th_{n_1,n_2}^c f_{b,n_1,n_2}(q)
$$
where $d_{n_1,n_2}$ and $\th_{n_1,n_2}$ are given by Equations \eqref{eq.dim}
and \eqref{eq.theta} respectively. The result is still a bivariate 
sequence of Laurent polynomials in~$q$. Since $d_{n_1,n_2}$ and $\th_{n_1,n_2}$ 
are proper $q$-hypergeometric (in the language of \cite{WZ}), it follows 
that the generators for the annihilator ideal of $f_{b,n_1,n_2}(q)$ can easily 
be translated to generators of the annihilator ideal of $F_{b,n_1,n_2}(q)$ 
and vice-versa. This is explained in detail in the Mathematica notebook 
\texttt{GroebnerBasis.CJ.2.3.nb} of \cite{Ko4}.

\appendix

\section{Generators for the recursion ideal of $f_{3,n_1,n_2}(q)$}
\lbl{sec.app1}


{\small
\begin{eqnarray*}
P_1&=&-q^6M_1^3M_2(q^3M_1-1)(qM_2-1)(q^4M_1M_2-1)(q^5M_1M_2^2-1)F_1L_1^2\\
&&-q^4M_1M_2(q^2M_1-1)(q^2M_2-1)(M_1-M_2)(q^4M_1M_2-1)F_6L_1L_2\\
&&+q^6M_1M_2^3(qM_1-1)(q^3M_2-1)(q^4M_1M_2-1)(q^5M_1^2M_2-1)\tau(F_1)L_2^2\\
&&+M_2(q^2M_1-1)(qM_2-1)(q^3M_1M_2-1)(q^5M_1^2M_2-1)F_7L_1\\
&&+(-M_1)(qM_1-1)(q^2M_2-1)(q^3M_1M_2-1)(q^5M_1M_2^2-1)\tau(F_7)L_2\\
&&+q^3(qM_1-1)(qM_2-1)(M_1-M_2)(q^2M_1M_2-1)F_2\\[1ex]
P_2&=&-q^{21}M_1^3M_2^7(q^4M_2-1)(q^5M_1M_2-1)(q^5M_1M_2^2-1)F_1L_2^3\\
&&+q^7M_1^2M_2^3(q^2M_1-1)(q^2M_2-1)(q^4M_1M_2-1)(q^6M_1M_2^2-1)F_8L_1L_2\\
&&+q^6M_2^3(q^3M_2-1)(q^4M_1M_2-1)F_{16}L_2^2\\
&&+q^8M_1^2M_2^2(q^2M_1-1)(qM_2-1)(q^3M_1M_2-1)(q^6M_1M_2^2-1)F_3L_1\\
&&+(q^2M_2-1)(q^3M_1M_2-1)(q^5M_1M_2^2-1)F_{14}L_2\\
&&-q^2(qM_2-1)(q^2M_1M_2-1)F_9\\[1ex]
P_3&=&-q^{17}M_1^3M_2^5(q^2M_1-1)(q^3M_2-1)(q^5M_1M_2-1)(q^5M_1M_2^2-1)F_1L_1L_2^2\\
&&+q^5M_1M_2^2(q^2M_1-1)(q^2M_2-1)(q^4M_1M_2-1)(q^6M_1M_2^2-1)F_{12}L_1L_2\\
&&+q^6M_2^3(qM_1-1)(q^3M_2-1)(q^4M_1M_2-1)F_{13}L_2^2\\
&&+q^7M_1M_2^2(q^2M_1-1)(qM_2-1)(q^3M_1M_2-1)F_{11}L_1\\
&&-(qM_1-1)(q^2M_2-1)(q^3M_1M_2-1)(q^5M_1M_2^2-1)F_{15}L_2\\
&&+q^2(qM_1-1)(qM_2-1)(q^2M_1M_2-1)(q^6M_1M_2^2-1)F_4 
\end{eqnarray*}
\begin{eqnarray*}
Q_1&=&q^{47}M_1^5M_2^{10}(q^6M_2-1)(q^7M_1M_2-1)(q^6M_1M_2^2-1)(q^7M_1M_2^2-1)F_5L_2^5\\
&&-q^{24}M_1^2M_2^6(q^5M_2-1)(q^6M_1M_2-1)(q^6M_1M_2^2-1)F_{20}L_2^4\\
&&+q^9M_2^3(q^4M_2-1)(q^5M_1M_2-1)(q^9M_1M_2^2-1)F_{23}L_2^3\\
&&+(q^3M_2-1)(q^4M_1M_2-1)(q^7M_1M_2^2-1)F_{24}L_2^2\\
&&+q^2(q^2M_2-1)(q^3M_1M_2-1)(q^{10}M_1M_2^2-1)F_{18}L_2\\
&&-q^6(qM_2-1)(q^2M_1M_2-1)(q^9M_1M_2^2-1)(q^{10}M_1M_2^2-1)F_{10}\\[1ex]
Q_2&=&q^8M_1^3M_2^2(q-1)(q^2M_1-1)(q^2M_1^2+qM_1+1)(qM_2-1)(q^3M_1M_2-1)\\
&&\qquad\times(q^6M_1M_2^2-1)(q^7M_1M_2^2-1)(q^8M_1M_2^2-1)F_5L_1\\
&&-q^{37}M_1^5M_2^{10}(q^5M_2-1)(q^6M_1M_2-1)(q^6M_1M_2^2-1)F_8L_2^4\\
&&+q^{18}M_1^2M_2^6(q^4M_2-1)(q^5M_1M_2-1)F_{19}L_2^3\\
&&-q^6M_2^3(q^3M_2-1)(q^4M_1M_2-1)(q^7M_1M_2^2-1)F_{21}L_2^2\\
&&-(q^2M_2-1)(q^3M_1M_2-1)F_{22}L_2\\
&&+q^2(qM_2-1)(q^2M_1M_2-1)(q^8M_1M_2^2-1)F_{17}
\end{eqnarray*}
}
where $F_1,\dots,F_{24}$ are irreducible polynomials in $\BQ[q,M_1,M_2]$,
too large to print all of them here:

{\small

\begin{eqnarray*}
F_1&=&q^{18} M_1^6 M_2^6-q^{15} M_1^6 M_2^4-q^{15} M_1^4 M_2^6+q^{14} M_1^5 M_2^4-q^{14} M_1^4 M_2^5+q^{13} M_1^5 M_2^4-q^{12} M_1^5 M_2^4-q^{11} M_1^3 M_2^4\\
&&+q^{11} M_1^2 M_2^5-q^{10} M_1^5 M_2^2+q^{10} M_1^4 M_2^3+q^9 M_1^4 M_2^2-q^9 M_1^3 M_2^3+q^8 M_1^4 M_2^2+2 q^8 M_1^3 M_2^3+q^8 M_1^2 M_2^4\\
&&-q^7 M_1^4 M_2^2-q^7 M_1^3 M_2^3-q^6 M_1^2 M_2^2-q^4 M_1^2 M_2-q^4 M_2^3+q^3 M_1^3+q^3 M_1 M_2^2+q M_2-M_1\\
F_2&=&q^{19} M_1^6 M_2^6-q^{16} M_1^6 M_2^4-q^{16} M_1^4 M_2^6-q^{13} M_1^5 M_2^4-q^{13} M_1^4 M_2^5+q^{11} M_1^4 M_2^3+q^{11} M_1^3 M_2^4+q^{10} M_1^5 M_2^2\\
&&+q^{10} M_1^2 M_2^5+q^9 M_1^4 M_2^2+q^9 M_1^3 M_2^3+q^9 M_1^2 M_2^4-2 q^8 M_1^3 M_2^3+q^7 M_1^3 M_2^3-q^5 M_1^2 M_2^2-q^4 M_1^2 M_2\\
&&-q^4 M_1 M_2^2-q^3 M_1^3-q^3 M_1^2 M_2-q^3 M_1 M_2^2-q^3 M_2^3+q^2 M_1^2 M_2+q^2 M_1 M_2^2+M_1+M_2\\
& \vdots & \\
F_6&=&q^{22} M_1^7 M_2^7-q^{20} M_1^7 M_2^7-q^{18} M_1^7 M_2^6-q^{18} M_1^6 M_2^7+q^{15} M_1^7 M_2^4+q^{15} M_1^6 M_2^5+q^{15} M_1^5 M_2^6-2 q^{15} M_1^5 M_2^5\\
&& +q^{15} M_1^4 M_2^7+q^{14} M_1^6 M_2^4+q^{14} M_1^4 M_2^6+q^{13} M_1^5 M_2^5+q^{12} M_1^5 M_2^5-q^{12} M_1^4 M_2^4-q^{11} M_1^6 M_2^2-q^{11} M_1^2 M_2^6\\
&& -q^{10} M_1^4 M_2^4+q^9 M_1^4 M_2^3+q^9 M_1^3 M_2^4-q^8 M_1^5 M_2^2-3 q^8 M_1^4 M_2^3-3 q^8 M_1^3 M_2^4-q^8 M_1^2 M_2^5+q^7 M_1^4 M_2^3\\
&& +q^7 M_1^3 M_2^4+q^6 M_1^3 M_2^2+q^6 M_1^2 M_2^3+q^4 M_1^4+q^4 M_1^3 M_2+2 q^4 M_1^2 M_2^2+q^4 M_1 M_2^3+q^4 M_2^4-q^3 M_1^3 M_2\\
&& -q^3 M_1 M_2^3-q M_1^2-q M_1 M_2-q M_2^2+M_1 M_2\\
F_7&=&q^{24} M_1^8 M_2^7-q^{21} M_1^8 M_2^5-q^{21} M_1^6 M_2^7+q^{19} M_1^6 M_2^6-q^{18} M_1^7 M_2^5-q^{18} M_1^6 M_2^6+q^{16} M_1^5 M_2^5-q^{16} M_1^4 M_2^6\\
&& +q^{15} M_1^7 M_2^3+q^{15} M_1^4 M_2^6+q^{14} M_1^6 M_2^3+q^{14} M_1^5 M_2^4+2 q^{14} M_1^4 M_2^5-3 q^{13} M_1^5 M_2^4+q^{12} M_1^5 M_2^4\\
&& -q^{12} M_1^4 M_2^3+q^{11} M_1^3 M_2^4-q^{11} M_1^2 M_2^5+q^{10} M_1^5 M_2^2-q^{10} M_1^4 M_2^3-2 q^{10} M_1^4 M_2^2+q^9 M_1^3 M_2^3+q^9 M_1^2 M_2^4\\
&& -q^8 M_1^5 M_2+q^8 M_1^4 M_2^2-3 q^8 M_1^3 M_2^3-q^8 M_1^2 M_2^4+q^7 M_1^4 M_2^2+q^7 M_1^4+q^7 M_1^3 M_2^3-q^5 M_1^4+q^5 M_1^3 M_2\\
&& +q^4 M_1^2 M_2+q^4 M_2^3-q^3 M_1^3-q^3 M_1 M_2^2-q M_2+M_1\\
& \vdots & \\
\end{eqnarray*}
}
The complete list is available at \cite{Ko4}.

\section{The $q=1$ limit for the recursion ideal of $f_{3,n_1,n_2}(q)$}
\lbl{sec.app2}


Let $p_i=\e(P_i)$ denote the image of $P_i$ when $q=1$. Then, up to polynomial
factors of $M_1,M_2$, we have: 

{\small

\begin{eqnarray*}
p_1 &=&
M_1^2-L_2 M_1^2-L_1 M_1 M_2+L_2 M_1 M_2-L_1^2 M_1^4 M_2
+L_1 L_2 M_1^4 M_2-M_2^2+L_1 M_2^2-M_1^3 M_2^2+L_1^2 M_1^3 M_2^2
\\ & &
+L_2 M_1^3 M_2^2-L_1 L_2 M_1^3 M_2^2+M_1^2 M_2^3-L_1 M_1^2 M_2^3+L_1 L_2 M_1^2 M_2^3-L_2^2 M_1^2 M_2^3 +L_1 M_1^5 M_2^3-L_1 L_2 M_1^5 M_2^3
\\ & &
-L_1 L_2 M_1 M_2^4+L_2^2 M_1 M_2^4+L_1 M_1^4 M_2^4-L_1^2 M_1^4 M_2^4-L_2 M_1^4 M_2^4+L_2^2 M_1^4 M_2^4-L_2 M_1^3 M_2^5+L_1 L_2 M_1^3 M_2^5
\\ & &
-L_1 M_1^6 M_2^5+L_1^2 M_1^6 M_2^5+L_2 M_1^5 M_2^6-L_2^2 M_1^5 M_2^6
\\
p_2 &=&
(-1+L_2) (M_1^2-M_1 M_2+M_2^2-L_1 M_1^2 M_2^3 +L_2 M_1^2 M_2^3+L_1 M_1^5 M_2^3-L_2 M_1 M_2^4-L_2 M_1^4 M_2^4+L_2 M_2^5
\\ & &
+L_2 M_1^3 M_2^5-L_2 M_1^2 M_2^6-L_2^2 M_1^4 M_2^7+L_2^2 M_1^3 M_2^8-L_2^2 M_1^5 M_2^9)
\\
p_3 &=&
(-1+L_2) 
(-1+M_1^2 M_2-M_1 M_2^2+L_1 M_1 M_2^2
\\ & &
-L_2 M_2^3-L_1 M_1^3 M_2^3+L_2 M_1^2 M_2^4+L_1 M_1^5 M_2^4-L_2 M_1^4 M_2^5+L_1 L_2 M_1^4 M_2^5-L_1 L_2 M_1^3 M_2^6+L_1 L_2 M_1^5 M_2^7)
\\[1ex]
q_1&=&
(-1+L_2)^2 (1+L_2 M_2^3) (1+L_2 M_1^3 M_2^3) (-1+L_2 M_1^2 M_2^4)
\\
q_2 &=&
-(-1+L_2) (1+L_2 M_2^3) (1+L_2 M_1^3 M_2^3) (-1+L_2 M_1^2 M_2^4)
\end{eqnarray*}
}

\subsection*{Acknowledgments}
The paper came into maturity following requests for explicit formulas
for the recursion of the colored Jones polynomial of a knot, during visits of 
the first author in the Max-Planck-Institut f\"ur Mathematik in 2009-2010, 
and during an Oberwolfach workshop 1033/2010 in August 2010. 
The authors met during the conference in honor of Doron Zeilberger's
60th birthday, held at Rutgers University in May 2010, and wish to thank
the organizers for their hospitality. The authors especially wish to 
thank Don Zagier and Doron Zeilberger for their interest, their
encouragement and for stimulating conversations.

\bibliographystyle{hamsalpha}\bibliography{biblio}

\providecommand{\bysame}{\leavevmode\hbox to3em{\hrulefill}\thinspace}
\providecommand{\href}[2]{#2}
\providecommand{\eprint}{\begingroup \urlstyle{rm}\Url}
\begin{thebibliography}{FGLM93}

\bibitem[BN]{B-N}
Dror Bar-Natan, \emph{{K}notatlas}, \url{http://katlas.org}.

\bibitem[Buc65]{Bu}
Bruno Buchberger, \emph{Ein algorithmus zum auffinden der basiselemente des
  restklassenrings nach einem nulldimensionalen polynomideal}, Ph.D. thesis,
  University of Innsbruck, Austria, 1965.

\bibitem[CS98]{CS}
Fr{\'e}d{\'e}ric Chyzak and Bruno Salvy, \emph{Non-commutative elimination in
  {O}re algebras proves multivariate identities}, J. Symbolic Comput.
  \textbf{26} (1998), no.~2, 187--227.

\bibitem[FGLM93]{FGLM}
J.~C. Faug{\`e}re, P.~Gianni, D.~Lazard, and T.~Mora, \emph{Efficient
  computation of zero-dimensional {G}r\"obner bases by change of ordering}, J.
  Symbolic Comput. \textbf{16} (1993), no.~4, 329--344.

\bibitem[FH91]{FH}
William Fulton and Joe Harris, \emph{Representation theory}, Graduate Texts in
  Mathematics, vol. 129, Springer-Verlag, New York, 1991, A first course,
  Readings in Mathematics.

\bibitem[Gar]{Ga2}
Stavros Garoufalidis, \emph{The {$\mathrm{SL}_N$} character variety of the
  trefoil}, Preprint 2011.

\bibitem[Gar04]{Ga1}
\bysame, \emph{On the characteristic and deformation varieties of a knot},
  Proceedings of the {C}asson {F}est, Geom. Topol. Monogr., vol.~7, Geom.
  Topol. Publ., Coventry, 2004, pp.~291--309 (electronic).

\bibitem[GL05]{GL}
Stavros Garoufalidis and Thang T.~Q. L{\^e}, \emph{The colored {J}ones function
  is {$q$}-holonomic}, Geom. Topol. \textbf{9} (2005), 1253--1293 (electronic).

\bibitem[GVa]{GV2}
Stavros Garoufalidis and Thao Vuong, \emph{The {D}egree and {S}lope
  {C}onjectures for simple lie algebras}, Preprint 2011.

\bibitem[GVb]{GV1}
\bysame, \emph{The {$\mathfrak{sl}_3$} colored {Jones} polynomial of the
  trefoil}, \eprint{arXiv:1010.3147}, Preprint 2010.

\bibitem[Hum78]{Hu}
James~E. Humphreys, \emph{Introduction to {L}ie algebras and representation
  theory}, Graduate Texts in Mathematics, vol.~9, Springer-Verlag, New York,
  1978, Second printing, revised.

\bibitem[Jan96]{Ja}
Jens~Carsten Jantzen, \emph{Lectures on quantum groups}, Graduate Studies in
  Mathematics, vol.~6, American Mathematical Society, Providence, RI, 1996.

\bibitem[Jon87]{Jo}
V.~F.~R. Jones, \emph{Hecke algebra representations of braid groups and link
  polynomials}, Ann. of Math. (2) \textbf{126} (1987), no.~2, 335--388.

\bibitem[Kas78]{Ks}
Masaki Kashiwara, \emph{On the holonomic systems of linear differential
  equations. {II}}, Invent. Math. \textbf{49} (1978), no.~2, 121--135.

\bibitem[Kau09a]{Ka2}
Manuel Kauers, \emph{{\tt Guess} {M}athematica software}, 2009,
  \url{http://www.risc.uni-linz.ac.at}.

\bibitem[Kau09b]{Ka1}
\bysame, \emph{Guessing handbook}, Tech. Report 09-07, RISC Report Series,
  Johannes Kepler University Linz, Austria, 2009.

\bibitem[Kou09]{Ko1}
Christoph Koutschan, \emph{Advanced applications of the holonomic systems
  approach}, Ph.D. thesis, RISC, Johannes Kepler University Linz, Austria,
  2009.

\bibitem[Kou10a]{Ko3}
\bysame, \emph{{\tt HolonomicFunctions} {M}athematica software}, 2010,
  \url{http://www.risc.uni-linz.ac.at}.

\bibitem[Kou10b]{Ko2}
\bysame, \emph{Holonomicfunctions: (user's guide)}, Tech. Report 10-01, RISC
  Report Series, Johannes Kepler University Linz, 2010.

\bibitem[Kou11]{Ko4}
\bysame, \emph{{M}athematica notebooks {G}uess{I}deal.cj.2.3.nb and
  {G}roebner{B}asis.cj.2.3.nb}, 2011,
  \url{http://www.math.gatech.edu/~stavros/publications.html}.

\bibitem[Ore33]{Or}
Oystein Ore, \emph{Theory of non-commutative polynomials}, Ann. of Math. (2)
  \textbf{34} (1933), no.~3, 480--508.

\bibitem[PR]{PR2}
Peter Paule and Axel Riese, \emph{{\tt qZeil} {M}athematica software},
  \url{http://www.risc.uni-linz.ac.at}.

\bibitem[PR97]{PR1}
\bysame, \emph{A {M}athematica {$q$}-analogue of {Z}eilberger's algorithm based
  on an algebraically motivated approach to {$q$}-hypergeometric telescoping},
  Special functions, {$q$}-series and related topics ({T}oronto, {ON}, 1995),
  Fields Inst. Commun., vol.~14, Amer. Math. Soc., Providence, RI, 1997,
  pp.~179--210.

\bibitem[PWZ96]{PWZ}
Marko Petkov{\v{s}}ek, Herbert~S. Wilf, and Doron Zeilberger, \emph{{$A=B$}}, A
  K Peters Ltd., Wellesley, MA, 1996, With a foreword by Donald E. Knuth, With
  a separately available computer disk.

\bibitem[RJ93]{JR}
Marc Rosso and Vaughan Jones, \emph{On the invariants of torus knots derived
  from quantum groups}, J. Knot Theory Ramifications \textbf{2} (1993), no.~1,
  97--112.

\bibitem[Sab93]{Sa}
Claude Sabbah, \emph{Syst\`emes holonomes d'\'equations aux
  {$q$}-diff\'erences}, {$D$}-modules and microlocal geometry ({L}isbon, 1990),
  de Gruyter, Berlin, 1993, pp.~125--147.

\bibitem[Tur88]{Tu1}
V.~G. Turaev, \emph{The {Y}ang-{B}axter equation and invariants of links},
  Invent. Math. \textbf{92} (1988), no.~3, 527--553.

\bibitem[Tur94]{Tu2}
\bysame, \emph{Quantum invariants of knots and 3-manifolds}, de Gruyter Studies
  in Mathematics, vol.~18, Walter de Gruyter \& Co., Berlin, 1994.

\bibitem[WZ92]{WZ}
Herbert~S. Wilf and Doron Zeilberger, \emph{An algorithmic proof theory for
  hypergeometric (ordinary and ``{$q$}'') multisum/integral identities},
  Invent. Math. \textbf{108} (1992), no.~3, 575--633.

\bibitem[Zei90]{Z}
Doron Zeilberger, \emph{A holonomic systems approach to special functions
  identities}, J. Comput. Appl. Math. \textbf{32} (1990), no.~3, 321--368.

\end{thebibliography}
\end{document}